\documentclass[12pt]{amsart}

\usepackage{amssymb,latexsym}

\usepackage{color}

\usepackage{enumerate}

\usepackage[T1]{fontenc}

 \usepackage[french,english]{babel}

\makeatletter

\@namedef{subjclassname@2010}{

  \textup{2010} Mathematics Subject Classification}

\makeatother
\newtheorem{thm}{Theorem}[section]

\newtheorem{lem}[thm]{Lemma}

\theoremstyle{definition}

\numberwithin{equation}{section}

\newcommand{\hk}{\mathcal{H}_k}
\newcommand{\A}{\mathcal{A}}

\newcommand{\im}{\textup{Im}}

\newcommand{\ep}{\varepsilon}

\renewcommand{\le}{\leqslant}
\renewcommand{\leq}{\leqslant}
\renewcommand{\ge}{\geqslant}
\renewcommand{\geq}{\geqslant}

\newcommand{\newabstract}[1]{%
  \par\bigskip
  \csname otherlanguage*\endcsname{#1}%
  \csname captions#1\endcsname
  \item[\hskip\labelsep\scshape\abstractname.]
}

\begin{document}

\baselineskip=17pt

\title[]{An omega result for the least negative Hecke eigenvalue}

\author{Youness Lamzouri}

\address{
Universit\'e de Lorraine, CNRS, IECL,
F-54000 Nancy, France}

\email{youness.lamzouri@univ-lorraine.fr}


\begin{abstract} 
We establish the existence of many holomorphic Hecke eigenforms $f$ of large
weight $k$ for the full modular group, for which the least positive integer $n_f$ such that $\lambda_f(n_f)<0$ satisfies $n_f \ge (\log k)^{1-o(1)}.$ This is believed to be best possible up to the $o(1)$ term in the exponent, and improves on a result of Kowalski, Lau, Soundararajan and Wu,
who showed that, when restricted to primes, the least prime $p$ such that
$\lambda_f(p)<0$ can be as large as $(\log k)^{1/2+o(1)}$. We also discuss an extension of our result to primitive holomorphic cusp forms of weight
$k$ and squarefree level $N\geq 1$.

\end{abstract}


\maketitle


\section{Introduction}

Let $k$ be a large even integer, and denote by $\hk$ the set of holomorphic Hecke cusp forms of weight $k$ for the full modular group $\textup{SL}(2,\mathbb{Z})$. Then, $\hk$ is an orthonormal basis for the space of holomorphic cusp forms of weight $k$ for $\textup{SL}(2,\mathbb{Z})$ and we have 
\begin{equation}\label{Eq:CardinalityHK}
    |\hk|=\frac{k}{12}+O\left(k^{2/3}\right).
\end{equation}
Given $f\in \hk$, its Fourier expansion can be written as
$$ f(z)=\sum_{n=1}^{\infty}\lambda_f(n) n^{(k-1)/2}e(nz), \ \  \text{ for } \ \ \im(z)>0,$$
where $e(z):=e^{2\pi iz}.$
The $\lambda_f(n)$ are the normalized
eigenvalues of the Hecke operators, and satisfy the Hecke relations
\begin{equation}\label{Eq:Hecke}
 \lambda_f(m)\lambda_f(n)=\sum_{d|(m,n)}\lambda_f\left(\frac{mn}{d^2}\right),
\end{equation}
for all $m, n\geq 1$. In particular, $\lambda_f$ is a real-valued multiplicative function of $n$ and  satisfies Deligne's bound
\begin{equation}\label{Deligne}
\left|\lambda_f(n)\right|\leq d(n),
\end{equation}
where $d(n)$ is the divisor function. Hence, for every prime $p$, there exists a unique angle $\theta_f(p)\in [0, \pi]$ such that
\begin{equation}\label{Eq:Angles}
   \lambda_f(p)=2\cos(\theta_f(p)). 
\end{equation}


In \cite{KLSW}, Kowalski, Lau, Soundararajan and Wu studied the signs of the sequence $\lambda_f(n)$.
Their results show a strong analogy between  the first negative Hecke eigenvalue and the problem of the least quadratic non-residue, which has a rich history in analytic number theory. 
Let $n_f$ be the smallest positive integer $n$ such that $\lambda_f(n)<0$. The best known bound for $n_f$ is due to Matom\"aki \cite{Ma12}, who improved the results of \cite{KLSW} by showing that 
$$n_f\ll k^{3/4}.$$
 This is probably far from the truth, since it is known that $n_f\ll (\log k)^2$ under the assumption of the generalized Riemann hypothesis (GRH). 
In analogy with the least quadratic non-residue, a folklore conjecture asserts that the correct order of magnitude for the maximum of $n_f$ over $f\in \hk$ should be $(\log k)^{1+o(1)}$.
Moreover, a simple probabilistic heuristic argument, in which
the values $\{\lambda_f(p)\}_{p\ \mathrm{prime}}$ are modeled by independent
Sato--Tate distributed random variables, suggests the more precise prediction that the maximal order of
$n_f$, as $f$ varies over $\mathcal H_k$, is likely to be
$\asymp  \log k \log\log k.$

In the case of the least quadratic non-residue, it is known that the analogue of the $(\log k)^{1+o(1)}$ lower bound is attained infinitely often. Indeed, using an easy argument, based on the law of quadratic reciprocity and Linnik's theorem on the least prime in an arithmetic progression, Fridlender \cite{Fr49} and Sali\'e \cite{Sa49} independently showed the existence of infinitely many primes $p$ such that the least quadratic non-residue modulo $p$ is $\gg \log p$. This lower bound was sharpened to $\gg\log p \log_3 p$ by Graham and Ringrose \cite{GrRi90}, where here and throughout $\log_j$ denotes the $j$-th iterate of the natural logarithm for $j\ge 1$. Assuming GRH, Montgomery \cite{Mo71} improved this bound to $\gg \log p \log\log p$ and conjectured that this is best possible. 

In the case of the first negative Hecke eigenvalue, the only known omega result  is due to Kowalski, Lau, Soundararajan and Wu \cite{KLSW}, and states that for $\gg |\hk|^{1-\varepsilon}$ cusp forms $f\in \hk$, the least prime $p_f$ such that $\lambda_f(p_f)<0$ satisfies $p_f\gg \sqrt{\log k\log\log k}.$ One should note that this result does not give information on $n_f$, since the latter might be strictly smaller than $p_f$, due to the fact that $\lambda_f$ is not completely multiplicative\footnote{However, by the multiplicativity of $\lambda_f$, we know that $n_f$ is a power of a prime.}.

In this paper, we fill this gap, showing the existence of $\gg |\hk|^{1-\varepsilon}$ cusp forms $f\in \hk$ such that $n_f\gg \log k/\log\log k.$  
\begin{thm}\label{Thm:Main}
Let $k$ be a large even integer. There are $\displaystyle{\gg|\hk|\exp\bigg(-3\frac{\log k\log_3 k}{(\log\log k)^2}\bigg)}$ cusp forms $f\in \hk$ such that $$n_f\gg \frac{\log k}{\log\log k}.$$
\end{thm}

Our approach differs from that of Kowalski, Lau,
Soundararajan and Wu \cite{KLSW}, who used a construction of Barton,
Montgomery and Vaaler \cite{BMV} of multidimensional trigonometric
minorants and majorants for the characteristic function of a box in  $[0,\pi]^{\pi(z)}$, to study the joint distribution of the
angles $(\theta_f(p))_{p\le z}$. 
In contrast, our method is inspired by ideas of Gonek and Montgomery \cite{GoMo1, GoMo2},
who built on earlier work of Tur\'an \cite{Tu60} to obtain a sharp localized form of Kronecker’s theorem
concerning inhomogeneous Diophantine approximation.

We now discuss an extension of our result to the set
$\mathcal H_k^{*}(N)$ of primitive holomorphic cusp forms of weight $k$,
squarefree level $N\geq 1$, and trivial nebentypus. Analogously to
\eqref{Eq:CardinalityHK}, it is known that
\begin{equation}\label{Eq:CardinalityHKN}
\lvert \mathcal H_k^{*}(N) \rvert
\sim \frac{k-1}{12}\,\varphi(N),
\end{equation}
as $kN\to\infty$, where $\varphi$ is Euler’s totient function. In
this setting, Kowalski, Lau, Soundararajan and Wu \cite{KLSW} showed that
there exist $\gg \lvert \mathcal H_k^{*}(N) \rvert^{1-\varepsilon}$ forms
$f\in\mathcal H_k^{*}(N)$ for which the least prime $p$ such that
$p\nmid N$ and $\lambda_f(p)<0$ satisfies
$
p \gg \sqrt{(\log (kN))(\log\log (kN))}.
$
In analogy with Theorem \ref{Thm:Main}, we prove the following result.

\begin{thm}\label{Thm:Main2}
Let $k\ge 2$ be an even integer and let $N\ge 1$ be squarefree, with
$kN$ sufficiently large. There are
$
\gg \lvert \mathcal H_k^{*}(N) \rvert
\exp\big(-3 \log(kN)\log_3(kN)/(\log_2(kN))^2\big)
$
forms $f\in\mathcal H_k^{*}(N)$ for which the least positive integer $n$
such that $(n,N)=1$ and $\lambda_f(n)<0$ satisfies
$$
n \gg \frac{ \log (kN)}{\log \log (kN)}.
$$
\end{thm}

\subsection*{Acknowledgements} The author is partially supported by a junior chair of the Institut Universitaire de France. This paper was written while the author was visiting the Max Planck Institute for Mathematics in Bonn, which he thanks for its warm hospitality and excellent working conditions. The author also thanks Sarvagya Jain for pointing out a refinement of the argument that led to the slight improvement in Theorem \ref{Thm:Main}, and Alia Hamieh for comments on this work.


\section{Preliminary results}

\subsection{A sharply localized trigonometric polynomial} 
In what follows we make use of a trigonometric
polynomial with a very large peak, whose construction goes back to
Chebyshev and was later exploited by Gonek and Montgomery  \cite{GoMo1, GoMo2}. 

\begin{lem}\label{Lem:approx}

Let $L$ be a positive integer, and suppose that $0<\delta\le \frac12$.
There exist a trigonometric polynomial $h:\mathbb{R}\to \mathbb{C}$ of the form
$$
h(\theta)=\sum_{\ell=0}^{L} c_\ell e(-\ell\theta),
$$
such that

\begin{itemize}
    \item[1.]  $\max_{\theta\in \mathbb{R}} |h(\theta)|=h(0)=1$ and
$
|h(\theta)|\le 2e^{-\pi L\delta}$
for $
\delta\le \theta\le 1-\delta$. Moreover,
for all $\theta\in \mathbb{R}$ we have $|h(\theta)|= |h(-\theta)|$. 
\item[2.] We have
\begin{equation}\label{Eq:L2} \int_0^1 |h(\theta)|^2 d\theta \geq \frac{1}{L+1},
\end{equation}
and 
\begin{equation}\label{Eq:L2SatoTate}\int_0^{1/2} |h(\theta)|^2 \sin(2\pi \theta)^2 d\theta \gg \frac{1}{L^3}.
\end{equation}
    
\end{itemize}

\end{lem}
\begin{proof}
    Part 1 corresponds to Lemma 7 of Gonek--Montgomery \cite{GoMo2} where they chose
    $$ h(\theta)= e^{-i\pi L\theta} \frac{T_L\big(\cos(\pi\theta)/\cos(\pi\delta)\big)}{T_L\big(1/\cos(\pi\delta)\big)},$$
    where $T_L$ is the $L$-th Chebyshev polynomial of the first kind.  We now establish part 2. First, by Parseval's Theorem and the Cauchy--Schwarz inequality we have
    $$ \int_0^1|h(\theta)|^2d\theta=\sum_{\ell=0}^{L}|c_\ell|^2\geq \frac{1}{L+1}\bigg|\sum_{\ell=0}^{L}c_\ell\bigg|^2= \frac{h(0)^2}{L+1},$$
    which establishes \eqref{Eq:L2} since $h(0)=1$. 
    
    Now we prove \eqref{Eq:L2SatoTate}. 
Let $\varepsilon>0$ be a small parameter to be chosen, and put
$$
G_\varepsilon=[\varepsilon, 1/2-\varepsilon]
\quad\text{and}\quad
E_\varepsilon=[0,1/2]\setminus G_\varepsilon.
$$
Then for all $\theta\in G_\varepsilon$, we have
$
\sin(2\pi \theta)^2\gg \varepsilon^2.
$
Moreover, by \eqref{Eq:L2} we get
\begin{align*}
\int_{E_\varepsilon}|h(\theta)|^2\,d\theta
\le \int_{E_\varepsilon} 1\,d\theta= 2\varepsilon
\leq 2\varepsilon(L+1)\int_0^1|h(\theta)|^2\,d\theta.
\end{align*}
We now choose $\varepsilon=\frac{1}{16L}$ which implies that
$$
\int_{E_\varepsilon}|h(\theta)|^2\,d\theta
\le \frac14\int_0^1|h(\theta)|^2\,d\theta
=\frac12\int_0^{1/2}|h(\theta)|^2\,d\theta,
$$
since $|h(-\theta)|=|h(\theta)|$ and $h$ is $1$-periodic.
Hence, by \eqref{Eq:L2} we deduce that
$$
\int_{G_\varepsilon}|h(\theta)|^2\,d\theta
\ge \frac12\int_0^{1/2}|h(\theta)|^2\,d\theta
\ge \frac{1}{4(L+1)}.
$$
Therefore, we obtain
$$
\int_0^{1/2}|h(\theta)|^2\sin(2\pi \theta)^2\,d\theta
\ge \int_{G_\varepsilon}|h(\theta)|^2\sin(2\pi \theta)^2\,d\theta
\gg \varepsilon^2\int_{G_\varepsilon}|h(\theta)|^2\,d\theta
\gg \frac{1}{L^3},
$$
as desired.
\end{proof}

\subsection{Basic facts on Hecke eigenvalues}
Since $\hk=\mathcal{H}_k^*(N)$ when $N=1$, we recall some well-known facts concerning the Hecke eigenvalues $\lambda_f(n)$ in the general case $f\in\mathcal{H}_k^*(N)$. The Hecke relations now read
\begin{equation}\label{Eq:Hecke+}
\lambda_f(m)\lambda_f(n)=\sum_{\substack{d|(m,n)\\ (d, N)=1}}\lambda_f\left(\frac{mn}{d^2}\right),
\end{equation}
for all $m, n\geq 1$.  
By \eqref{Eq:Angles} and  \eqref{Eq:Hecke+} it follows that for all positive integers $m$ and all prime numbers  $p\nmid N$ we have 
\begin{equation}\label{Eq:HeckePowers}
\lambda_f(p^m)= \frac{\sin((m+1)\theta_f(p))}{\sin(\theta_f(p))}.
\end{equation}
We will require the following lemma, which is a  consequence of the Petersson trace formula. To this end, we define the harmonic weight 
$$ 
\omega_f:= \frac{\Gamma(k-1)}{(4\pi)^{k-1}\langle f, f \rangle}\frac{N}{\varphi(N)}, 
$$
where $\langle f, f \rangle$ is the Petersson norm of $f$. Using well-known bounds for $\langle f, f \rangle$ we have 
\begin{equation}\label{Eq:BoundHarmonic}
\omega_f\ll \frac{(\log (kN))^2}{kN}.
\end{equation}
\begin{lem}\label{lem:Petersson}
For all positive integers $m$ we have
$$
\sum_{f\in\mathcal{H}_k^*(N)} \omega_f \lambda_f(m)
=\mathbf{1}_{m=1}+O\left(m^{1/3}(kN)^{-5/6}\right).
$$
\end{lem}

\begin{proof}
This follows directly from \cite[Corollary 2.10]{ILS00}. See also the last
displayed equation on page 402 of \cite{KLSW}.
\end{proof}

\section{Proof of Theorems \ref{Thm:Main} and \ref{Thm:Main2}}
For $\theta\in\mathbb{R}$, we define
\begin{equation}\label{Eq:DefunctionG}
g(\theta)=\bigg|h\Big(\frac{\theta}{2\pi}\Big)\bigg|^2,
\end{equation}
where $h$ is the trigonometric polynomial in Lemma \ref{Lem:approx}.
Then we have $0\le g(\theta)\le 1$ for all $\theta\in \mathbb{R}$,
$\max_{\theta\in \mathbb{R}} g(\theta)=1$, $g$ is even and
$$
g(\theta)=\sum_{|\ell|\le L}b_\ell e^{i\ell\theta},
$$
where
$$
b_\ell
=\sum_{\substack{0\le \ell_1,\ell_2\le L\\ \ell_2-\ell_1=\ell}}
c_{\ell_1}\overline{c_{\ell_2}}.
$$

Recall that the Chebyshev functions $\{X_n\}_{n\ge 0}$, defined by
$$
X_n(\theta)=\frac{\sin((n+1)\theta)}{\sin\theta}
$$
for $\theta\in[0,\pi]$, form an orthonormal basis of
$L^2([0,\pi],\mu_{ST})$, where $\mu_{ST}=(2/\pi) (\sin \theta)^2d\theta$ is the Sato--Tate measure on $[0, \pi]$.  We can decompose $g$ in the basis $\{X_\ell\}_{\ell\ge 0}$ as follows
$$
g(\theta)=\sum_{\ell=0}^{\infty} a_\ell X_\ell(\theta),
$$
where
$$
a_\ell=\int_0^\pi g(\theta)X_\ell(\theta)\,d\mu_{ST}
=\frac{2}{\pi}\int_0^\pi g(\theta)X_\ell(\theta)(\sin\theta)^2\,d\theta.
$$
Since $g$ is an even trigonometric polynomial of degree $L$, it follows
from the definition of $X_\ell$ that $a_\ell=0$ for $\ell>L$. Hence we can write 
\begin{equation}\label{Eq:ChebyshevExpansion}
g(\theta)=\sum_{\ell=0}^{L} a_\ell X_\ell(\theta).   
\end{equation}
Moreover, by the Cauchy--Schwarz inequality, we have:
\begin{align}
|a_\ell|
&\le
\left(\int_0^\pi |g(\theta)|^2\,d\mu_{ST}\right)^{1/2}
\left(\int_0^\pi X_\ell(\theta)^2\,d\mu_{ST}\right)^{1/2} \nonumber\\
&\le
\left(\int_0^\pi 1\,d\mu_{ST}\right)^{1/2}
\left(\int_0^\pi X_\ell(\theta)^2\,d\mu_{ST}\right)^{1/2}
=1,
\end{align}
since $|g(\theta)|\le 1$ and $\{X_\ell\}_\ell$ is an orthonormal basis of
$L^2([0,\pi],\mu_{ST})$.

We now have all the necessary ingredients to prove Theorems \ref{Thm:Main} and \ref{Thm:Main2}.
\begin{proof}[Proof of Theorem \ref{Thm:Main}]
Let $2\leq z\leq \log k$ be a real number to be chosen. For each prime $p\leq z$ we let $\delta_p= (\log p)/(16 \log \log k)$ and $L_p=\lfloor 4 
(\log\log  k)^2/\log p \rfloor+1$. Let $h_p$ denote the trigonometric polynomial in Lemma \ref{Lem:approx}  with parameters $\delta=\delta_p$ and $L=L_p$, and let $g_p(\theta)=|h_p(\theta/(2\pi))|^2$. By \eqref{Eq:ChebyshevExpansion} we have 
\begin{equation}\label{Eq:ChebyshevExpansion2}
 g_p(\theta)= \sum_{\ell=0}^{L_p} a_{\ell, p} X_\ell(\theta), 
 \end{equation}
where $|a_{\ell, p}|\leq 1$. We also put $\varepsilon=4(\log k)^{-\pi/2}$.   
Then it follows from Lemma \ref{Lem:approx} that if
$\theta\in 
[2\pi \delta_p,\pi]
$
we have $0\leq g_p(\theta)\le 4e^{-2\pi L_p\delta_p}\leq \varepsilon$.
For $f\in\hk$ we define
\begin{equation}\label{Eq:DefG}
G(f):=\prod_{p\le z} g_p(\theta_f(p))
-\varepsilon\sum_{\substack{q\le z\\ q\text{ prime}}}
\prod_{\substack{p\le z\\ p\ne q}} g_p(\theta_f(p)).
\end{equation}
Then, note that if
$
\theta_f(q)\in [2\pi \delta_q,\pi]
$
for some prime $q\le z$, then
$$
G(f)
\le
\left(g_q(\theta_f(q))-\varepsilon\right)\prod_{\substack{p\le z\\ p\ne q}} g_p(\theta_f(p))
\le 0.
$$
On the other hand, for all $f\in\hk$ we have $G(f)\le 1$ since $0\leq g_p(\theta)\leq 1$ for all primes $p$ and all $\theta\in \mathbb{R}$.

Let $\A$ be the set of $f\in \hk$ such that $\theta_f(p)\in[0,2\pi \delta_p)$ for all primes $p\le z$,
and denote by $\mathbf{1}_{\A}$ the indicator function of $\A$.
Then we have
\begin{equation}\label{Eq:Compare1G}
\sum_{f\in\mathcal{H}_k} \omega_f \mathbf{1}_{\A}(f)
\ge
\sum_{f\in\mathcal{H}_k} \omega_f G(f).
\end{equation}
Furthermore, writing $p_j$ as the $j$th prime number and putting $J=\pi(z)$, we obtain by \eqref{Eq:ChebyshevExpansion2}
\begin{align*}
G(f)
&=
\prod_{p\le z}
\left(\sum_{\ell=0}^{L_p} a_{\ell, p} X_\ell(\theta_f(p))\right)
-
\varepsilon
\sum_{q\le z}
\prod_{\substack{p\le z\\ p\ne q}}
\left(\sum_{\ell=0}^{L_p} a_{\ell, p} X_\ell(\theta_f(p))\right)\\
&=
\sum_{\substack{0\le \ell_j\le L_{p_j}\\ 1\le j\le J}}
\prod_{j=1}^{J}
a_{\ell_j, p_j}X_{\ell_j}(\theta_f(p_j))
-
\varepsilon
\sum_{m=1}^{J}
\sum_{\substack{0\le \ell_j\le L_{p_j}\\ 1\le j\le J\\ j\ne m}}
\prod_{\substack{j=1\\ j\ne m}}^{J}
a_{\ell_j, p_j}X_{\ell_j}(\theta_f(p_j)).
\end{align*}
We now use \eqref{Eq:HeckePowers} which gives
$$
X_{\ell_1}(\theta_f(p_1))\cdots X_{\ell_J}(\theta_f(p_J))
=\lambda_f\left(\prod_{j=1}^{J} p_j^{\ell_j}\right).
$$
Inserting this above, and using that $|a_{\ell, p}|\le 1$ for all primes $p$ and integers $\ell$ gives
\begin{align}\label{Eq:EstimateAverageG}
\sum_{f\in\hk} \omega_f G(f)
&= \sum_{\substack{0\le \ell_j\le L_{p_j}\\ 1\le j\le J}}
\prod_{j=1}^{J}
a_{\ell_j, p_j}\sum_{f\in\hk} \omega_f \lambda_f\left(\prod_{j=1}^{J} p_j^{\ell_j}\right)
\nonumber\\
& \quad \quad \quad \quad -
\varepsilon
\sum_{m=1}^{J}
\sum_{\substack{0\le \ell_j\le L_{p_j}\\ 1\le j\le J\\ j\ne m}}
\prod_{\substack{j=1\\ j\ne m}}^{J}
a_{\ell_j, p_j} \sum_{f\in\hk} \omega_f \lambda_f\bigg(\prod_{\substack{j=1\\ j\ne m}}^{J} p_j^{\ell_j}\bigg)
\nonumber\\
&= \prod_{p\leq z} a_{0, p}
+O\Bigg(\varepsilon \sum_{q\leq z} \prod_{\substack{p\leq z\\ p\neq q}}a_{0, p}+ k^{-5/6}\prod_{p\le z}\big((L_p+1)p^{L_p/3}\big)
\Bigg),
\end{align}
since $J\varepsilon\leq 1$.
We now choose
$$
z=\frac{\log  k}{4\log \log k},
$$
so that $J \sim  \log k/(4(\log\log k)^2)$. Moreover, by \eqref{Eq:L2SatoTate} and the definition of $a_{0, p}$ we have 
\begin{align*}
a_{0, p}&=\int_0^\pi g_p(\theta)\,d\mu_{ST}
=\frac{2}{\pi}\int_0^\pi
\left|h_p\left(\frac{\theta}{2\pi}\right)\right|^2(\sin\theta)^2\,d\theta\\
&=4\int_0^{1/2}|h_p(t)|^2\sin(2\pi t)^2\,dt \gg \frac{1}{L_p^3} \gg \frac{1}{(\log\log k)^6}.
\end{align*}
Using these estimates in \eqref{Eq:EstimateAverageG} and recalling that $\ep=4 (\log k)^{-\pi/2}$ we obtain
$$
\sum_{f\in\hk} \omega_f G(f)
= (1+o(1))\prod_{p\leq z} a_{0, p}
\gg 
\exp\left(-2\frac{ \log k\log_3 k}{(\log_2 k)^2}\right).
$$
Next we combine this with the inequality \eqref{Eq:Compare1G} and the estimates \eqref{Eq:CardinalityHK} and \eqref{Eq:BoundHarmonic} to derive 
\begin{equation}\label{Eq:LowerBoundSetA}
|\mathcal{A}| \gg |\hk| \exp\left(-3\frac{ \log k\log_3 k}{(\log_2 k)^2}\right).
\end{equation}

Now let $f\in \A$ and suppose that $p^m\leq z$ for some prime $p$ and positive integer $m$. Then $p\leq z$ and $m\leq (\log z)/\log p\leq (\log \log k)/\log p.$ Therefore, by \eqref{Eq:HeckePowers} we get
$$ \lambda_f(p^m) = \frac{\sin((m+1)\theta_f(p))}{\sin(\theta_f(p))}>0, 
$$
since
$0\leq \theta_f(p)\leq 2\pi \delta_p$ (by our assumption that $f\in \A$), which implies that $$0\leq (m+1) \theta_f(p)\leq 4\pi\delta_p \frac{\log\log k}{\log p}= \frac{\pi}{4}. $$ This concludes the proof.
\end{proof}

\begin{proof}[Proof of Theorem \ref{Thm:Main2}]
The proof of Theorem \ref{Thm:Main2} follows along the exact same lines as that of
Theorem \ref{Thm:Main}, with the parameters $z=\log(kN)/(4\log\log(kN))$ and
$\varepsilon=4(\log(kN))^{-\pi/2}$, and with the function $G$ defined in \eqref{Eq:DefG} replaced by
$$
\widetilde{G}(f)
:=
\prod_{\substack{p\le z\\ p\nmid N}} \widetilde g_p(\theta_f(p))
-
\varepsilon
\sum_{\substack{q\le z\\ q\ \mathrm{prime}\\ q\nmid N}}
\prod_{\substack{p\le z\\ p\ne q\\ p\nmid N}} \widetilde g_p(\theta_f(p)),
$$   
where $\widetilde g_p(\theta)=|\widetilde h_p(\theta/(2\pi))|^2$, and $\widetilde h_p$ is the trigonometric polynomial of Lemma \ref{Lem:approx}
with parameters
$\delta=\widetilde \delta_p=(\log p)/(16\log\log(kN))$ and
$L=\widetilde L_p=\lfloor 4(\log\log(kN))^2/\log p\rfloor+1$.

\end{proof}

\end{document}